\documentclass[11pt,a4paper]{amsart}
\usepackage[T1]{fontenc}
\usepackage{lmodern}
\usepackage{microtype}
\usepackage[a4paper,textwidth=150mm,textheight=240mm,centering]{geometry}
\usepackage{mathtools,amssymb}
\usepackage{enumitem}
\usepackage[hidelinks,pdfusetitle]{hyperref}
\usepackage{bookmark}
\hypersetup{pdftitle={Covering dimension and geometric dimension of CAT spaces},pdfsubject={Research manuscript},pdfkeywords={CAT spaces, geometric dimension, covering dimension, local homology, free abelian groups}}
\setlist[enumerate]{leftmargin=*,itemsep=3pt,topsep=5pt}
\allowdisplaybreaks[1]

\newtheorem*{conjecture}{Conjecture}
\newtheorem*{stepranstheorem}{Theorem (Stepr\={a}ns)}
\newtheorem{maintheorem}{Theorem}

\newtheorem{theorem}{Theorem}[section]
\newtheorem{proposition}[theorem]{Proposition}
\newtheorem{lemma}[theorem]{Lemma}

\theoremstyle{definition}
\newtheorem{definition}[theorem]{Definition}
\theoremstyle{remark}
\newtheorem{remark}[theorem]{Remark}
\newtheorem*{remark*}{Remark}
\numberwithin{equation}{section}

\newcommand{\Z}{\mathbb Z}
\newcommand{\R}{\mathbb R}

\newcommand{\Sn}{\mathbb S}
\newcommand{\CAT}{\operatorname{CAT}}
\newcommand{\gdim}{\operatorname{dim}_{\!G}}

\newcommand{\cdim}{\operatorname{cd}}
\newcommand{\mesh}{\operatorname{mesh}}
\newcommand{\Size}{\operatorname{size}}
\newcommand{\Hom}{\operatorname{Hom}}
\newcommand{\Ext}{\operatorname{Ext}}

\newcommand{\sd}{\operatorname{sd}}
\newcommand{\Bdd}{\mathcal B}
\newcommand{\redH}{\widetilde H}

\newcommand{\into}{\hookrightarrow}

\title[Dimension of CAT spaces]{Dimension of CAT spaces}
\author{Alexander Lytchak}
\date{Working manuscript --- 5 September 2026}
\subjclass[2020]{53C23, 54F45, 55N10, 20K20}
\keywords{CAT spaces, dimension, free abelian groups}

\begin{document}
\begin{abstract}
We prove that the covering dimension of a metric space with curvature bounded above equals its geometric dimension, without a separability assumption. This resolves a conjecture formulated by Kleiner that goes back to Gromov's work.
\end{abstract}
\maketitle

\section{Introduction}

The geometric dimension $\gdim X$ of a metric space $X$ with curvature bounded above in the sense of Alexandrov was introduced by Bruce Kleiner in \cite{Kleiner}. He proved that it is the supremum of the dimensions of Euclidean balls that admit topological embeddings into $X$. The geometric dimension also coincides with the supremum of the covering dimensions $\dim K$ of compact subsets $K\subset X$.
Kleiner formulated the following conjecture in \cite[p.~412]{Kleiner}, tracing its origin to Gromov's work \cite[p.~133]{Gromov}:

\begin{conjecture}
For every metric space with curvature bounded above, the covering dimension coincides with the geometric dimension.
\end{conjecture}

The locally compact case was settled in \cite[Theorem~A]{Kleiner}; as observed in \cite[p.~412]{Kleiner}, the proof also applies to separable spaces. The conjecture was subsequently restated in \cite[Conjecture~14.19]{AKP}.
Our main result resolves this conjecture in full generality:

\begin{maintheorem}\label{thm:main}
Let $X$ be a CBA space. Then
\[
                         \dim X=\gdim X.
\]
\end{maintheorem}

Kleiner's results imply $\dim X\ge\gdim X$. His homological characterization of geometric dimension, together with the universal coefficient theorem and Theorem~\ref{thm:anr-cohomological-dimension} below, also gives the upper bound $\dim X\le\gdim X+1$; see Proposition~\ref{prop:dimension-criterion}. Theorem~\ref{thm:anr-cohomological-dimension} states the equality $\dim T=\cdim T$ between covering dimension and cohomological dimension for metrizable ANRs. The separable case is proved in \cite[Corollary~2.12]{DydakKoyama}, while the nonseparable case turns out to be a direct consequence of \cite[Theorem~0.2]{CenceljDydakMitraVavpetic}.

The proof of equality in Theorem~\ref{thm:main} reduces to establishing freeness of top-dimensional relative homology groups. This is the more technical contribution of the paper:

\begin{maintheorem}\label{thm:free}
Let $X$ be a CBA space, and assume that $\gdim X\le n$, where $n\ge0$ is an integer. For every open pair $U\subset V\subset X$, the group
\[
                              H_n(V,U)
\]
is free abelian.
\end{maintheorem}

Unlike the general ANR results used below, Theorem~\ref{thm:free} requires geometric input: top-dimensional relative homology need not be free even for a two-dimensional compact absolute retract. See Sklyarenko's example in \cite[Example~2.7, p.~7]{MelikhovShchepin} and the discussion at the end of Subsection~\ref{subsec:homological-criterion}. For local versus relative freeness, see Remark~\ref{rem:local-freeness}.

All homology groups below are singular homology groups. All homology and cohomology groups, and all chain groups, have integer coefficients, which are omitted from the notation.

Although the freeness of every countable subgroup of $H_n(V,U)$ in Theorem~\ref{thm:free} may be deduced from Kleiner's results \cite{Kleiner}, proving freeness of the whole group is much subtler.
We use the supports of top-dimensional cycles, as in \cite{KleinerLeeb,BestvinaKleinerSageev,Huang} and \cite[Section~4]{LytchakStadler}, to detect a class by its images in spaces of directions. The essential step is to construct an integer-valued homogeneous norm on each top-dimensional relative homology group. We then apply Stepr\={a}ns's theorem \cite{Steprans}, which states that an abelian group admitting a discrete homogeneous norm is free. This generalizes the bounded-function form of N\"obeling's theorem \cite{Nobeling}: for every set $I$, the group
\begin{equation}\label{eq:intro-bounded}
 \Bdd(I)=\bigl\{f:I\longrightarrow\Z:\sup_{i\in I}|f(i)|<\infty\bigr\}
\end{equation}
is free abelian. 

In Section~\ref{sec:prelim}, we recall the basic properties of spaces with upper curvature bounds needed in the proof. Section~\ref{sec:criterion} collects the relevant topological results: the equality of covering and cohomological dimensions for metrizable ANRs, the relation between homological and cohomological dimensions, and the localization of top-dimensional homology classes. Combining these results, Section~\ref{sec:reduction} reduces Theorem~\ref{thm:main} to Theorem~\ref{thm:free} and the proof of the latter to the case of $\CAT(1)$ spaces.

Section~\ref{sec:induction} contains the core of the paper, the proof of Theorem~\ref{thm:free}. Localization embeds a top-dimensional relative homology group into a direct product of top-dimensional homology groups of spaces of directions. The latter groups are free by induction, but a subgroup of a direct product of free abelian groups need not be free. To overcome this difficulty, we introduce a class size derived from the $\ell^1$-norm on sufficiently small chains and construct homogeneous norms on the homology groups controlled by this size. The key estimate bounds the sizes of all directional images uniformly in terms of the size of the original class. At this quantitative step, the curvature bound enters only through the elementary logarithmic estimate of Lemma~\ref{lem:log-lipschitz}; see Proposition~\ref{prop:directional-size}. The resulting norm is discrete, so Stepr\={a}ns's theorem completes the proof.

\subsection*{Acknowledgements}
The author was supported in part by the German Research Foundation under Germany's Excellence Strategy -- EXC-2047/1 -- 390685813.
This paper arose in close collaboration with ChatGPT, an AI assistant developed by OpenAI. The collaboration involved exploring proof strategies, refining arguments, locating references, and drafting and revising the manuscript. The author takes full responsibility for the mathematical content and the final text.

\section{Preliminaries}\label{sec:prelim}

\subsection{Topological notions}
For a metric space $T$ and an integer $n\ge0$, the inequality $\dim T\le n$ means that every open cover of $T$ has a locally finite open refinement of multiplicity at most $n+1$. We put $\dim\varnothing=-1$ and write $\dim T=\infty$ if no finite bound exists; see \cite[Chapters~3 and~4]{Engelking}.  Covering dimension is monotone under inclusion.

A metric space $T$ is an \emph{absolute neighborhood retract}, or ANR, if every realization of $T$ as a closed subset of a metric space admits a retraction from an open neighborhood onto $T$.

Throughout the paper, simplicial complexes are endowed with the metric topology on their geometric realizations. Every such realization is a metrizable ANR, and every CW complex or metrizable ANR has the homotopy type of such a realization; see \cite[pp.~245 and~247]{Mardesic} and \cite{Dowker}.

\subsection{CBA spaces}
We assume familiarity with the basic geometry of spaces with upper curvature bounds and refer to \cite{AKP,BH,Kleiner} for background. Here we fix notation. Throughout the paper, every $\CAT(\kappa)$ space is assumed to be metrically complete. Geodesic completeness is not assumed.

For $\kappa\in\R$, put $D_\kappa=\infty$ if $\kappa\le0$ and $D_\kappa=\pi/\sqrt\kappa$ if $\kappa>0$. In a $\CAT(\kappa)$ space, points at distance less than $D_\kappa$ are joined by a unique geodesic, which depends continuously on its endpoints \cite[II.1.4]{BH}. When $\kappa>0$, we do not require geodesics between points at distance at least $D_\kappa$. In particular, $\CAT(1)$ spaces need not be connected.

We say that a metric space $X$ has \emph{curvature bounded above}, or is a \emph{CBA space}, if every point $x\in X$ has a closed neighborhood that is $\CAT(\kappa)$ with its induced metric; the bound $\kappa$ may depend on $x$. A CBA space need not be globally complete; the closed $\CAT(\kappa)$ neighborhoods in this definition make it locally complete. Every open subset of a CBA space is an ANR \cite[Theorem~3.2]{Kramer}.

After rescaling and, if necessary, increasing the curvature bound, every global $\CAT(\kappa)$ space can be regarded as a $\CAT(1)$ space \cite[II.1.12]{BH}. These operations preserve the topology and the angular metrics on spaces of directions.

For $x$ in a CBA space $X$, the space of directions $\Sigma_xX$ is the completion of the space of geodesic directions at $x$ with its angular metric. $\Sigma_xX$ is a $\CAT(1)$ space.

Let $X$ be a $\CAT(1)$ space and $x\in X$. Every open ball $B_R(x)$ with $0<R\leq \pi$ is contractible by the radial geodesic homotopy \cite[II.1.4]{BH}. The \emph{angular logarithmic map}
\[
 \log_x:B_\pi(x)\setminus\{x\}\longrightarrow\Sigma_xX
\]
sends $y$ to the direction of the geodesic $xy$. The restriction of $\log_x$ to $B_R(x)\setminus\{x\}$ is a homotopy equivalence for every $0<R\le\pi$. For $R<\pi/2$, this is \cite[Theorem~3.5]{Kramer}; larger radii reduce to this case by radially shrinking the punctured ball into a smaller one.

The contractibility of $B_\pi(x)$, excision, the boundary homomorphism, and the logarithmic homotopy equivalence define, for every $m\ge1$, a canonical isomorphism
\begin{equation}\label{eq:local-link}
 \ell_x:H_m(X,X\setminus\{x\})
       \xrightarrow{\ \cong\ }\redH_{m-1}(\Sigma_xX).
\end{equation}
It is the composition
\[
\begin{aligned}
 H_m(X,X\setminus\{x\})
   &\xrightarrow{\ \mathrm{exc}^{-1}\ }
       H_m(B_\pi(x),B_\pi(x)\setminus\{x\})\\
   &\xrightarrow{\ \partial\ }
       \redH_{m-1}(B_\pi(x)\setminus\{x\})
     \xrightarrow{\ (\log_x)_*\ }\redH_{m-1}(\Sigma_xX).
\end{aligned}
\]
Here $\mathrm{exc}^{-1}$ is the inverse of the inclusion-induced excision isomorphism. A chain $b$ in $B_\pi(x)$ whose boundary avoids $x$ represents a local class, and $\ell_x$ sends this class to $[(\log_x)_\#\partial b]$.

The following coarse estimate follows from $\CAT(1)$ comparison and spherical geometry; cf. \cite[Section~3.1]{LytchakNagano}.

\begin{lemma}[Lipschitz control of the logarithmic map]\label{lem:log-lipschitz}
Let $X$ be a $\CAT(1)$ space, let $x\in X$, and let $0<r\le R\le1/2$. On the annulus
\[
 A_{r,R}(x)=\{y\in X:r\le d(x,y)\le R\},
\]
the angular logarithmic map is $4/r$-Lipschitz:
\begin{equation}\label{eq:log-lipschitz}
 d_{\Sigma_xX}(\log_x y,\log_x z)
 \le \frac{4}{r}\,d(y,z)
 \qquad(y,z\in A_{r,R}(x)).
\end{equation}
\end{lemma}

In the form used in Subsection~\ref{subsec:inductive-step}, if $0<r\le1/8$, $0<\delta\le r/2$, and $E\subset A_{r,r+\delta}(x)$ has diameter at most $\delta$, then
\begin{equation}\label{eq:log-small-sets}
                    \operatorname{diam}(\log_x E)\le\frac{4\delta}{r}.
\end{equation}

\subsection{Geometric dimension}\label{subsec:geometric-dimension}
For a CBA space $X$, geometric dimension is characterized by $\gdim\varnothing=-1$, by $\gdim X=0$ for nonempty discrete $X$, and the recursion
\begin{equation}\label{eq:dimension-drop}
 \gdim X\le n
 \quad\Longleftrightarrow\quad
 \gdim\Sigma_xX\le n-1 \quad\text{for every }x\in X
 \qquad(n\ge1).
\end{equation}

We use two conclusions of Kleiner's dimension theorem \cite[Theorem~A]{Kleiner}:
\begin{align}
 \gdim X
   &=\sup\{\dim K:K\subset X\text{ compact}\},\label{eq:compact-dimension}\\
 \gdim X
   &=\sup\{q:H_q(V,U)\ne0
                    \text{ for some open }U\subset V\subset X\}.
                    \label{eq:homological-dimension}
\end{align}
In particular, $\gdim X\le n$ implies
\begin{equation}\label{eq:open-vanishing}
 H_q(V,U)=0 \qquad(q>n)
\end{equation}
for every open pair. Here and below, ``open pair'' means that both sets are open in the ambient space; either may be empty.

\section{A dimension criterion for metrizable ANRs}\label{sec:criterion}

We formulate the topological reduction in terms of cohomological dimension. Throughout this section, cohomology has constant integer coefficients. We write $\check H^q$ for \v{C}ech cohomology and $H^q$ for singular cohomology.

\subsection{Cohomological dimension}
For a metrizable space $T$, the \emph{integral cohomological dimension of $T$} is
\begin{equation}\label{eq:cech-dimension-definition}
 \cdim T=\sup\bigl\{q\ge0:\check H^q(T,A)\ne0
                 \text{ for some closed }A\subset T\bigr\}.
\end{equation}
We take the supremum of the empty set to be $-1$ and write $\cdim T=\infty$ when no finite bound exists. Thus $\cdim \varnothing=-1$. Equivalently, for $n\ge0$,
\[
 \cdim T\le n
 \quad\Longleftrightarrow\quad
 \check H^q(T,A)=0
 \quad\text{for every closed }A\subset T\text{ and every }q>n.
\]
This is the usual closed-pair definition; see \cite[Section~1]{Dranishnikov}. It uses \v{C}ech cohomology, not singular cohomology.

\subsection{Absolute extensors and equality of dimension}\label{subsec:absolute-extensors}
We recall the notion of absolute extensors, which is of fundamental importance in dimension theory. We call a space $L$ an \emph{absolute extensor} of a space $T$ if every continuous map from a closed subset of $T$ to $L$ extends continuously over $T$, and we write $L\in\operatorname{AE}(T)$.

For metrizable $T$, this property depends only on the homotopy type of $L$ among metrizable ANRs and CW complexes. More precisely, if $L$ and $L'$ are two such spaces and $L\simeq L'$, then
\[
 L\in\operatorname{AE}(T)
 \quad\Longleftrightarrow\quad
 L'\in\operatorname{AE}(T).
\]
For these targets, an extension up to homotopy from a closed subset can be replaced by an exact extension: for ANRs, this is Borsuk's homotopy extension theorem \cite[p.~247]{Mardesic}; for CW complexes, see \cite[Corollaries~1.8 and~2.13 in the arXiv version]{DydakExtension}. In particular, passing between homotopy-equivalent CW and metric simplicial models does not change the absolute-extensor property.

By the sphere-extension characterization of covering dimension \cite[p.~252]{Mardesic}, for every metrizable space $T$ and every integer $n\ge0$,
\begin{equation} \label{eq: dimchar}
 \dim T\le n
 \quad\Longleftrightarrow\quad
 \Sn^n\in\operatorname{AE}(T).
\end{equation}

 Analogously, integral cohomological dimension admits an extension-theoretic characterization.   $K(\Z,n)$ is an Eilenberg--MacLane space represented by a simplicial complex with the metric topology: for $n\ge1$, it is connected, $\pi_n=\Z$, and all other homotopy groups vanish; for $n=0$, we take the discrete space $\Z$.  The extension characterization is \cite[Section~0, p.~1156]{Dydak-Walsh}
\begin{equation}\label{eq:cohomological-dimension-definition}
 \cdim T\le n
 \quad\Longleftrightarrow\quad
 K(\Z,n)\in \operatorname{AE}(T).
\end{equation}

By \cite[Lemma~3.1, p.~38]{Morita}, for every metrizable space $T$ we have the inequality
\begin{equation}\label{eq:cohomological-covering-bound}
                         \cdim T\le\dim T.
\end{equation}

The reverse inequality for separable ANRs is contained in \cite[Corollary~2.12]{DydakKoyama}. We verify this equality without a separability assumption:

\begin{theorem}[ANR dimension comparison]\label{thm:anr-cohomological-dimension}
For every metrizable ANR $T$,
\begin{equation}\label{eq:anr-cohomological-dimension}
                         \dim T=\cdim T.
\end{equation}
\end{theorem}

\begin{proof}
Suppose $\cdim T\le n$. For $n\ge 1$, apply \cite[Theorem~0.2]{CenceljDydakMitraVavpetic} to $L=\Sn^n$. This theorem applies to metrizable ANRs and says, in this case, that $SP(\Sn^n)\in \operatorname{AE}(T)$ implies $\Sn^n\in \operatorname{AE}(T)$. Here $SP$ is the infinite symmetric product, and the Dold--Thom theorem \cite[Theorem 4K.6]{Hatcher} identifies its homotopy type with $K(\Z,n)$. Thus $\Sn^n\in \operatorname{AE}(T)$, and \eqref{eq: dimchar} above gives $\dim T\le n$.  For $n=0$, the two-point space $\Sn^0$ is a retract of the discrete space $K(\Z,0)=\Z$, so the same characterization applies. Together with $\cdim T\le\dim T$, this proves the assertion. 
\end{proof}

The ANR hypothesis cannot simply be omitted: Dranishnikov constructed a compact metrizable space $X$ with $\dim X=\infty$ and $\cdim X\le3$; see \cite[Theorem~7.1]{Dranishnikov}.

\subsection{From singular cohomology to cohomological dimension}
The closed subsets occurring in~\eqref{eq:cech-dimension-definition} need not be ANRs, so one cannot simply replace \v{C}ech cohomology by singular cohomology in that definition. The following lemma gives the passage from the singular cohomology of open pairs used in this paper to cohomological dimension, via~\eqref{eq:cohomological-dimension-definition}.

\begin{lemma}\label{lem:singular-cohomological-bound}
Let $T$ be a metrizable ANR and let $n\ge1$. If
\[
                     H^{n+1}(T,U)=0
\]
for every open subset $U\subset T$, then $\cdim T\le n$.
\end{lemma}

\begin{proof}
Choose a simplicial model $E$ of $K(\Z,n)$. By our convention, $E$ is a metrizable ANR.

Let $A\subset T$ be closed and let $f:A\to E$ be continuous. By the neighborhood-extension characterization of metric ANRs, $f$ extends to a map $f_U:U\to E$ on an open neighborhood $U$ of $A$. Both $T$ and $U$ have CW homotopy type \cite[Theorem~8.1]{Kramer}. The representability of singular cohomology \cite[Theorem~4.57]{Hatcher} identifies their homotopy classes of maps to $E$ with their degree-$n$ integral cohomology. The exact sequence
\[
 H^n(T)\longrightarrow H^n(U)
                  \longrightarrow H^{n+1}(T,U)=0
\]
shows that there is a map $g:T\to E$ with $g|_U\simeq f_U$.

Since $g|_A\simeq f$, Borsuk's homotopy extension theorem \cite[p.~247]{Mardesic} deforms $g$ to a map whose restriction to $A$ is exactly $f$: the domain $T$ is metrizable, $A$ is closed, and the target $E$ is an ANR. Thus $E\in\operatorname{AE}(T)$, as required.
\end{proof}

\subsection{The homological criterion}\label{subsec:homological-criterion}
The universal coefficient theorem now isolates the possible one-dimensional gap and the role of top-dimensional freeness.

\begin{proposition}[Homological dimension criterion]\label{prop:dimension-criterion}
Let $T$ be a metrizable ANR and let $n\ge1$. Suppose that
\[
                    H_q(T,U)=0\qquad(q>n)
\]
for every open subset $U\subset T$. Then
\begin{equation}\label{eq:dimension-gap}
                     \dim T=\cdim T\le n+1.
\end{equation}
If, in addition, $H_n(T,U)$ is free abelian for every open $U\subset T$, then
\begin{equation}\label{eq:dimension-criterion-conclusion}
             \cdim T=\dim T\le n.
\end{equation}
\end{proposition}

\begin{proof}
 The universal coefficient theorem \cite[Section~3.1]{Hatcher} gives the exact sequence
\begin{equation}\label{eq:uct}
\begin{split}
 0\longrightarrow
 \Ext^1_{\Z}\bigl(H_{q-1}(T,U),\Z\bigr)
 \longrightarrow H^q(T,U)\\
 \longrightarrow\Hom_{\Z}\bigl(H_q(T,U),\Z\bigr)
 \longrightarrow0.
\end{split}
\end{equation}
The homological vanishing assumption makes both outer terms zero for $q>n+1$. In particular, $H^{n+2}(T,U)=0$ for every open $U$. Lemma~\ref{lem:singular-cohomological-bound}, with $n+1$ in place of $n$, and Theorem~\ref{thm:anr-cohomological-dimension} give~\eqref{eq:dimension-gap}.

If $H_n(T,U)$ is free, then $\Ext^1_{\Z}(H_n(T,U),\Z)=0$.  Thus~\eqref{eq:uct} also gives
\begin{equation}\label{eq:all-coeff-vanishing}
                     H^q(T,U)=0\qquad(q>n)
\end{equation}
for every open $U$. Applying Lemma~\ref{lem:singular-cohomological-bound} in degree $n+1$ yields $\cdim T\le n$. Theorem~\ref{thm:anr-cohomological-dimension} completes the proof.
\end{proof}

\begin{remark*}[Sklyarenko's compactum]
The freeness conclusion of Theorem~\ref{thm:free} does not extend to arbitrary finite-dimensional ANRs. Let $U$ be the mapping telescope
\[
 U=\operatorname{Tel}\bigl(\Sn^1\xrightarrow{\,2\,}\Sn^1
                    \xrightarrow{\,2\,}\Sn^1\xrightarrow{\,2\,}\cdots\bigr),
\]
where each arrow denotes a map of degree two, and let $S=U\cup\{\infty\}$ be its one-point compactification. Sklyarenko's compactum $S$ is a two-dimensional absolute retract whose local group $H_2(S,S\setminus\{\infty\})$ is isomorphic to $\Z[1/2]$ and hence is not free abelian; see \cite[Example~2.7, p.~7]{MelikhovShchepin}.
\end{remark*}

\subsection{Detection of top-dimensional classes by local homology}\label{sec:support}

Let $A$ be a subset of a metric space $T$ and let $\alpha\in H_m(T,A)$. For $x\in T\setminus A$, denote by $\alpha_x$ the image of $\alpha$ under the natural homomorphism
\[
 H_m(T,A)\longrightarrow H_m(T,T\setminus\{x\}).
\]
The \emph{support} of $\alpha$ is the set (cf. \cite[Definition~10]{LytchakStadler} and \cite{Huang})
\[
 \operatorname{spt}(\alpha)
   =\{x\in T\setminus A:\alpha_x\ne0\}.
\]
 It is the intersection of $T\setminus A$ with the carriers of all relative cycles representing $\alpha$. Thus it is closed in $T\setminus A$ and has compact closure in $T$. If $A$ is open, the support itself is compact. The support depends on the homology class, not on a chosen representative.

Let $n\ge1$ and suppose that $H_{n+1}(V,W)=0$ for every open pair $W\subset V\subset T$. Then every nonzero class $\alpha\in H_n(T,A)$ with $A\subset T$ \emph{open} has nonempty support. This follows from the relative argument in \cite[proof of Lemma~A-5, pp.~2342--2343]{Huang}: its Mayer--Vietoris step \cite[proof of Lemma~3-2, p.~2301]{Huang} uses only vanishing in degree $n+1$ and works with integer coefficients, although Huang uses coefficients in $\Z/2\Z$.

For spaces of homological dimension at most $n$, the assertion also follows directly from \cite[Lemma~11]{LytchakStadler}. Here homological dimension at most $n$ means that $H_q(V,W)=0$ for every open pair $W\subset V\subset T$ and every $q>n$. If the support is empty, this lemma gives a representative in every neighborhood of $A$; since $A$ is open, one may take that neighborhood to be $A$ itself, and the class is zero in $H_n(T,A)$. For arbitrary subsets $A$, representation in every neighborhood of $A$ does not imply representation in $A$, so openness cannot simply be omitted from the preceding assertion.

We reformulate this localization statement as follows:

\begin{lemma}\label{lem:local-detection}
Let $T$ be a metric space and $n\ge1$. Suppose
\[
                    H_{n+1}(V,W)=0
\]
for every open pair $W\subset V\subset T$. Then, for every open $U\subset T$, the natural homomorphism
\begin{equation}\label{eq:local-detection}
 E_U:H_n(T,U)\longrightarrow
          \prod_{x\in T\setminus U}H_n(T,T\setminus\{x\})
\end{equation}
is injective.
\end{lemma}

\begin{remark}[Local freeness and countability]\label{rem:local-freeness}
Under the hypotheses of Lemma~\ref{lem:local-detection}, suppose that all local groups $H_n(T,T\setminus\{x\})$ are free abelian. The embedding~\eqref{eq:local-detection} alone does not show that $H_n(T,U)$ is free: the countable direct product $\prod_{k=1}^{\infty}\Z$ is not free abelian~\cite{Baer}.

If $H_n(T,U)$ is countable, then its freeness as a countable subgroup of a direct product of free abelian groups follows from Specker's theorem \cite{Specker}. On the other hand, for a separable metrizable ANR $T$ and every open $U\subset T$, all groups $H_n(T,U)$ are countable, as follows from \cite[Theorem~1, p.~272]{Milnor}.

We are not aware of an $n$-dimensional metrizable ANR $T$ for which all local groups $H_n(T,T\setminus\{x\})$ are free abelian but $H_n(T,U)$ is not free for some open $U\subset T$.
\end{remark}

\section{Reduction}\label{sec:reduction}
Here we reduce Theorem A to Theorem B and the latter to its local version.

\begin{proposition}\label{prop:reduction}
Theorem~\ref{thm:free} implies Theorem~\ref{thm:main}. 
\end{proposition}

\begin{proof}
Kleiner's compact-subset characterization~\eqref{eq:compact-dimension} gives $\gdim X\le\dim X$.  The cases $\gdim X=-1,0,\infty$ are obvious.

Suppose $1\le n=\gdim X<\infty$. The space $X$ is a metrizable ANR. For every open $U\subset X$, equation~\eqref{eq:open-vanishing} gives $H_q(X,U)=0$ for $q>n$, and Theorem~\ref{thm:free} gives freeness of $H_n(X,U)$. Proposition~\ref{prop:dimension-criterion} therefore yields
\[
                      \dim X=\cdim X\le n.
\]
Together with the opposite inequality, this proves the assertion. 
\end{proof}

The local-to-global principle follows from paracompactness:

\begin{proposition}\label{prop:local-to-global}
Fix an integer $n\ge0$. If the assertion of Theorem~\ref{thm:free} in degree $n$ holds for every $\CAT(1)$ space, then it holds for every CBA space.
\end{proposition}

\begin{proof}
For $n=0$, $H_0(V,U)$ is free on the path components of $V$ disjoint from $U$. 

Let $n\ge1$, let $X$ be a CBA space with $\gdim X\le n$, and fix an open pair $U\subset V\subset X$. We may assume $V\ne\varnothing$. Every point of $V$ has a $\CAT(\kappa)$ neighborhood contained in $V$.
The bound $\kappa$ may vary from one neighborhood to another.

By paracompactness of metric spaces and the closed-shrinking theorem (see \cite[Theorem~41.4 and Lemma~41.6]{Munkres}), we find a locally finite family $(F_i)_{i\in I}$ of subsets closed in $V$, open sets $O_i\subset V$, and $\CAT(\kappa_i)$ neighborhoods $C_i\subset V$ such that
\begin{equation}\label{eq:local-closed-cover}
              V=\bigcup_{i\in I}F_i,
              \qquad F_i\subset O_i\subset\operatorname{int}_V C_i.
\end{equation}
In particular, $O_i$ is open in $C_i$.   

The inequality $\gdim C_i\le\gdim X\le n$ follows directly from~\eqref{eq:compact-dimension}. The $\CAT(\kappa_i)$ space $C_i$ becomes a $\CAT(1)$ space after rescaling, without changing its geometric dimension. Hence the validity of Theorem B for $\CAT(1)$ spaces implies that $H_n(U_i,W_i)$ is free for every pair of open subsets $W_i\subset U_i$ of $C_i$.

For each $i$, set $A_i=U\cup(V\setminus F_i)$, which is open in $V$. Since $O_i\cup A_i=V$, excision \cite[Theorem~2.20]{Hatcher} gives
\begin{equation}\label{eq:local-relative-groups}
 \mathcal H_i:=H_n(V,A_i)
       \cong H_n(O_i,O_i\cap A_i).
\end{equation}
The pair on the right is an open pair in $C_i$. Thus $\mathcal H_i$ is free abelian by the assumed $\CAT(1)$ case of Theorem~\ref{thm:free}.

The maps of pairs $(V,U)\to(V,A_i)$ define a homomorphism
\begin{equation}\label{eq:local-to-global-map}
 \mathcal R:H_n(V,U)\longrightarrow\bigoplus_{i\in I}\mathcal H_i.
\end{equation}
To check that its values lie in the direct sum (a priori, they lie in the direct product), represent a class $\alpha$ by a finite relative cycle $z$. Its carrier $|z|$, the union of the images of its simplices, is compact in $V$ and meets only finitely many $F_i$. For every other $i$, the chain $z$ is contained in $V\setminus F_i\subset A_i$, so its image in $\mathcal H_i$ is zero.

The homomorphism $\mathcal R$ is injective. If $\mathcal R(\alpha)=0$, choose, for any $x\in V\setminus U$, an index $i$ with $x\in F_i$. Then $A_i\subset V\setminus\{x\}$, and the local image of $\alpha$ factors through $\mathcal H_i$. Hence every local image is zero. Equation~\eqref{eq:open-vanishing} and Lemma~\ref{lem:local-detection}, applied to $(V,U)$, give $\alpha=0$.

The direct sum in~\eqref{eq:local-to-global-map} is free abelian. Since subgroups of free abelian groups are free, $H_n(V,U)$ is free as well.
\end{proof}

\begin{remark}[Local-to-global for ANRs]\label{rem:anr-local-to-global}
The proof of Proposition~\ref{prop:local-to-global} applies without change to metrizable ANRs, with the same homological vanishing assumption. More precisely, let $T$ be a metrizable ANR and let $n\ge1$. Assume that $H_{n+1}(V,W)=0$ for every open pair $W\subset V\subset T$. If every point of $T$ has an open neighborhood $O$ such that $H_n(V,W)$ is free abelian for every open pair $W\subset V\subset O$, then $H_n(V,U)$ is free abelian for every open pair $U\subset V\subset T$. The argument uses only paracompactness, excision, compact carriers, and Lemma~\ref{lem:local-detection}; no curvature bound is involved. The case $n=0$ is immediate.
\end{remark}

\section{Induction}\label{sec:induction}

By Proposition~\ref{prop:local-to-global}, it remains to prove Theorem~\ref{thm:free} for $\CAT(1)$ spaces. We introduce a size of a homology class that incorporates the subdivision scale and prove that it can be controlled under localization.

\subsection{Size of a class}\label{subsec:size}

A singular $m$-chain $z=\sum_\sigma a_\sigma\sigma$ is always finite, and we combine equal simplices so that the displayed simplex maps are distinct. Its \emph{carrier}, its $\ell^1$-norm, and its mesh are
\[
 |z|=\bigcup_{a_\sigma\ne0}\sigma(\Delta^m),\qquad
 \|z\|_1=\sum_\sigma|a_\sigma|,\qquad
 \mesh z=\max_{a_\sigma\ne0}\operatorname{diam}\sigma(\Delta^m).
\]
For the zero chain, the carrier is empty and $\|z\|_1=\mesh z=0$. The carrier is compact and should be distinguished from the support of the homology class. We call a chain $\delta$-small if its mesh is at most $\delta$.

Let $Y$ be a metric space and $U\subset Y$ be open. For a chain $z\in C_m(Y)$ with $\partial z\in C_{m-1}(U)$, equivalently $|\partial z|\subset U$, define its relative control $\rho_U(z)\in(0,1]$ by
\begin{equation}\label{eq:boundary-clearance}
\rho_U(z)=\min\{1,d\bigl(|\partial z|,Y\setminus U\bigr)\}.
\end{equation}
We use the convention that the distance
$d\bigl(|\partial z|,Y\setminus U\bigr)$ is $\infty$ if either subset is empty.
In particular, $\rho_U(z)=1$ for an absolute cycle.

Put $\delta_m=32^{-m}$ for $m\ge0$. We call such a chain $z$ \emph{admissible} if
\begin{equation}\label{eq:admissible-chain}
                      \mesh z\le\delta_m \cdot \rho_U(z).
\end{equation}
Thus the required mesh decreases as the boundary approaches $Y\setminus U$.

\begin{definition}\label{def:size}
For $\alpha\in H_m(Y,U)$, define
\begin{equation}\label{eq:class-size}
 \Size_m(\alpha;Y,U)
   =\inf\bigl\{\|z\|_1:
           z\text{ is an admissible chain representing }\alpha\bigr\}.
\end{equation}
The infimum is taken over chains $z\in C_m(Y)$ with $\partial z\in C_{m-1}(U)$ whose relative homology class is $\alpha$.
For absolute homology, we write
\[
 \Size_m(\alpha;Y)=\Size_m(\alpha;Y,\varnothing).
\]
\end{definition}

\begin{lemma}\label{lem:finite-size}
Every homology class has finite size, and the infimum in~\eqref{eq:class-size} is attained. The zero class has size zero, and every nonzero class has size at least one.
\end{lemma}

\begin{proof}
Choose a chain $z_0\in C_m(Y)$ with $\partial z_0\in C_{m-1}(U)$ representing $\alpha$. Barycentric subdivision preserves its relative homology class and satisfies
\begin{equation}\label{eq:subdivision}
 |\partial\sd^k z_0|\subset|\partial z_0|,
 \qquad \|\sd^k z_0\|_1\le((m+1)!)^k\|z_0\|_1;
\end{equation}
see \cite[Section~2.1]{Hatcher}. Thus $\rho_U(\sd^kz_0)\ge\rho_U(z_0)>0$. Since
\[
 \lim_{k\to\infty}\mesh(\sd^kz_0)=0,
\]
the chain $\sd^kz_0$ is admissible for all sufficiently large $k$. The set of norms of admissible chains representing $\alpha$ is therefore a nonempty set of nonnegative integers, hence has a minimum. Only the zero chain has norm zero.
\end{proof}

\subsection{Inductive step}\label{subsec:inductive-step}

Let $Y$ be a $\CAT(1)$ space and let $m\ge1$. For $U\subset Y$ open and $x\in Y\setminus U$, compose the natural map
$H_m(Y,U)\to H_m(Y,Y\setminus\{x\})$ with the isomorphism
$\ell_x:H_m(Y,Y\setminus\{x\})\to\redH_{m-1}(\Sigma_xY)$ from~\eqref{eq:local-link} and the canonical inclusion
$\redH_{m-1}(\Sigma_xY)\into H_{m-1}(\Sigma_xY)$.
This gives the \emph{directional homomorphism}
\begin{equation}\label{eq:directional-map}
 \mathcal D_x:H_m(Y,U)\longrightarrow H_{m-1}(\Sigma_xY).
\end{equation}
The map $\mathcal D_x$ vanishes on $\alpha$ precisely when $x\notin\operatorname{spt}(\alpha)$. If $\alpha$ is represented by a chain $b$ in $B_\pi(x)$ with $|\partial b|\subset U$, then $\mathcal D_x\alpha$ is represented by $(\log_x)_\#\partial b$.

If $\gdim Y\le m$, Lemma~\ref{lem:local-detection}, \eqref{eq:local-link}, and \eqref{eq:open-vanishing} show that the map
\[
 \mathcal D_U:H_m(Y,U)\longrightarrow
       \prod_{x\in Y\setminus U}H_{m-1}(\Sigma_xY)
\]
with coordinates $\mathcal D_x$ is injective.

\begin{proposition}[Size under passage to directions]\label{prop:directional-size}
Let $Y$ be a $\CAT(1)$ space, let $U\subset Y$ be open, and let $m\ge1$. For every $\alpha\in H_m(Y,U)$ and every $x\in Y\setminus U$,
\begin{equation}\label{eq:directional-size}
 \Size_{m-1}(\mathcal D_x\alpha;\Sigma_xY)
                  \le(m+1)\Size_m(\alpha;Y,U).
\end{equation}
\end{proposition}

\begin{proof}
Fix $x\in Y\setminus U$. By Lemma~\ref{lem:finite-size}, choose an admissible representative $z$ of $\alpha$ with $\|z\|_1=\Size_m(\alpha;Y,U)$. Set $\rho=\rho_U(z)\in(0,1]$. By admissibility, $\mesh z\le\delta$, where
\[
 \delta:=\delta_m\rho\le\rho/32.
\]
Choose the auxiliary radius
\[
 r:=\rho/8\in(0,1/8],\qquad r\ge4\delta.
\]
By the definition of $\rho$, either $\partial z=0$ as a chain in $C_{m-1}(U)$ or $d(x,|\partial z|)\ge\rho$. Thus the carrier of $\partial z$ avoids $B_\rho(x)$. Since $r+\delta\le5\rho/32<\rho$, we have
\[
 |\partial z|\cap B_{r+\delta}(x)=\varnothing.
\]

Let $z_{x,r}$ be the subchain of $z$ consisting of those terms whose simplex images meet $B_r(x)$. Since $\mesh z\le\delta$, each of these images is contained in $B_{r+\delta}(x)$. Hence
\[
 |z_{x,r}|\subset B_{r+\delta}(x),\qquad
 |z-z_{x,r}|\cap B_r(x)=\varnothing,
 \qquad \|z_{x,r}\|_1\le\|z\|_1.
\]
In the decomposition $z=z_{x,r}+(z-z_{x,r})$, the images of simplices appearing in different summands can meet only in the annulus
\[
 \mathcal A_{r,r+\delta}(x):=\{y\in Y:r\le d(x,y)<r+\delta\}.
\]
Every face of a simplex in $z-z_{x,r}$ avoids $B_r(x)$. Therefore, any singular $(m-1)$-simplex whose image meets $B_r(x)$ has the same coefficient in $\partial z_{x,r}$ as in $\partial z$, namely zero. Since the boundary of $z_{x,r}$ is also contained in $B_{r+\delta}(x)$, it follows that
\[
 |\partial z_{x,r}|\subset\mathcal A_{r,r+\delta}(x).
\]
In particular, $\partial z_{x,r}$ avoids $x$, so $z_{x,r}$ represents a local homology class in $H_m(Y,Y\setminus\{x\})$.

The chain $z-z_{x,r}$ is contained in $Y\setminus\{x\}$ and is therefore zero in the relative chain group $C_m(Y,Y\setminus\{x\})$. Thus $z_{x,r}$ and $z$ represent the same local homology class $\alpha_x$.

By~\eqref{eq:local-link}, the class $\mathcal D_x\alpha\in H_{m-1}(\Sigma_xY)$ is represented by
\[
 b_x:=(\log_x)_\#\partial z_{x,r}\in C_{m-1}(\Sigma_xY).
\]
Since $z_{x,r}$ is a subchain of $z$, and every face has image contained in that of its simplex,
\[
 \mesh z_{x,r}\le\delta,\qquad \mesh(\partial z_{x,r})\le\delta.
\]
Pushforward is nonexpanding for the $\ell^1$-norm, and an $m$-simplex has $m+1$ faces. Hence
\[
 \|b_x\|_1\le\|\partial z_{x,r}\|_1
 \le(m+1)\|z_{x,r}\|_1
 \le(m+1)\|z\|_1
 =(m+1)\Size_m(\alpha;Y,U).
\]
To prove~\eqref{eq:directional-size}, it remains to check that $b_x$ is admissible.

Since $0<r\le1/8$ and $r+\delta\le5/32<1/2$, Lemma~\ref{lem:log-lipschitz} applies to the annulus $\mathcal A_{r,r+\delta}(x)$ and gives
\[
 \mesh b_x
 \le\frac{4}{r}\mesh(\partial z_{x,r})
 \le\frac{4\delta}{r}
 =32\delta_m=\delta_{m-1}.
\]
As an absolute cycle, $b_x$ has relative control one and is therefore admissible. Together with the preceding norm estimate, this proves~\eqref{eq:directional-size}.
\end{proof}

\subsection{Discrete norms and freeness}\label{subsec:coordinates}

We use a strengthening of N\"obeling's theorem due to Stepr\={a}ns, which involves the following concept.

 A \emph{homogeneous norm} on an abelian group $A$ is a function $\nu:A\to[0,\infty)$ such that
\[
 \nu(a)=0\ \Longleftrightarrow\ a=0,\qquad
 \nu(a+b)\le\nu(a)+\nu(b),\qquad
 \nu(ka)=|k|\nu(a)\quad(k\in\Z).
\]
It is \emph{discrete} if $\inf_{a\ne0}\nu(a)>0$. The following theorem is due to Stepr\={a}ns \cite{Steprans}; see also \cite[Theorem~3.1]{BraunSauer}.

\begin{stepranstheorem}
An abelian group with a discrete homogeneous norm is free.
\end{stepranstheorem}

In particular, an integer-valued homogeneous norm implies freeness. For a family $(A_i,\nu_i)_{i\in I}$ with such norms, the \emph{bounded product}
\[
 \prod_{i\in I}^{\mathrm{bd}}(A_i,\nu_i)
   =\left\{(a_i)\in\prod_{i\in I}A_i:
                     \sup_{i\in I}\nu_i(a_i)<\infty\right\}
\]
is free abelian: the supremum is again an integer-valued homogeneous norm. For $A_i=\Z$ and $\nu_i(a)=|a|$, this recovers N\"obeling's bounded-function theorem.

We do not apply this theorem directly to $\Size_m$, which need not satisfy the homogeneity property. Instead, the size is used to control the following genuine norms.

\begin{proposition}[A controlled discrete norm]\label{prop:bounded-coordinates}
Let $Y$ be a $\CAT(1)$ space with $\gdim Y\le m$, where $m\ge0$, and let $U\subset Y$ be open. The group $H_m(Y,U)$ admits an integer-valued homogeneous norm $\nu_m^{Y,U}$ satisfying
\begin{equation}\label{eq:coordinate-size-bound}
 \nu_m^{Y,U}(\alpha)\le(m+1)!\Size_m(\alpha;Y,U).
\end{equation}
In particular, $H_m(Y,U)$ is free abelian.
\end{proposition}

\begin{proof}
We construct the norms by induction on $m$. For $m=0$, let $\mathcal C(Y,U)$ be the set of path components of $Y$ that do not meet $U$. Under the canonical identification
\[
 H_0(Y,U)=\bigoplus_{C\in\mathcal C(Y,U)}\Z,
\]
we define the norm by
\[
 \nu_0^{Y,U}((a_C))=\sum_C|a_C|=\Size_0((a_C);Y,U).
\]
This proves the base case, including the required bound.

Suppose the norms have been constructed in degree $m-1$, where $m\ge1$. If $U=Y$, then $H_m(Y,Y)=0$, and we take $\nu_m^{Y,Y}=0$. Assume henceforth that $U\ne Y$. By~\eqref{eq:dimension-drop}, every $\Sigma_xY$ has geometric dimension at most $m-1$, so the inductive hypothesis applies. For every $\alpha\in H_m(Y,U)$ and $x\notin U$, it gives, together with Proposition~\ref{prop:directional-size},
\[
 \nu_{m-1}^{\Sigma_xY,\varnothing}(\mathcal D_x\alpha)
 \le m!\Size_{m-1}(\mathcal D_x\alpha;\Sigma_xY)
 \le(m+1)!\Size_m(\alpha;Y,U).
\]
The last bound is finite by Lemma~\ref{lem:finite-size} and independent of $x$. Thus the injective homomorphism $\mathcal D_U$ introduced in Subsection~\ref{subsec:inductive-step} takes values in the bounded product
\begin{equation}\label{eq:bounded-directional-map}
 \mathcal D_U:H_m(Y,U)\longrightarrow
 \prod_{x\in Y\setminus U}^{\mathrm{bd}}
 \bigl(H_{m-1}(\Sigma_xY),
                     \nu_{m-1}^{\Sigma_xY,\varnothing}\bigr).
\end{equation}

Pull back the supremum norm, setting
\begin{equation}\label{eq:inductive-norm}
 \nu_m^{Y,U}(\alpha)
   =\sup_{x\in Y\setminus U}
       \nu_{m-1}^{\Sigma_xY,\varnothing}(\mathcal D_x\alpha).
\end{equation}
The preceding estimate proves~\eqref{eq:coordinate-size-bound}. The terms in this supremum are nonnegative integers and are bounded above. Their supremum is therefore an integer. Since $\mathcal D_U$ is injective, the norm is positive on nonzero classes. The triangle inequality and integer homogeneity follow from the corresponding properties of the norms on the factors and the linearity of the directional maps. Thus $\nu_m^{Y,U}$ is an integer-valued homogeneous norm, as required.

Every nonzero class has norm at least one. Stepr\={a}ns's theorem therefore implies that $H_m(Y,U)$ is free abelian.
\end{proof}

\begin{proof}[Proof of Theorems~\ref{thm:free} and~\ref{thm:main}]
First let $Y$ be a $\CAT(1)$ space with $\gdim Y\le n$. Proposition~\ref{prop:bounded-coordinates} gives freeness of $H_n(Y,U)$ for every open $U\subset Y$.

For an open pair $U\subset V\subset Y$, the exact sequence of the triple contains
\[
 H_{n+1}(Y,V)\longrightarrow H_n(V,U)
                              \longrightarrow H_n(Y,U).
\]
The first group vanishes by~\eqref{eq:open-vanishing}. Hence $H_n(V,U)$ is a subgroup of a free abelian group and is free. This proves the $\CAT(1)$ case. Proposition~\ref{prop:local-to-global} gives Theorem~\ref{thm:free} for all CBA spaces. 
Proposition~\ref{prop:reduction} then proves Theorem~\ref{thm:main}.
\end{proof}

\end{document}